\documentclass[twoside,leqno,
]{amsart}
\usepackage{latexsym,amssymb,amsfonts,euscript}
\usepackage[dvips]{graphicx}

\newtheorem{thm}{Theorem}[section]
 \newtheorem{cor}[thm]{Corollary}
 
 \newtheorem{prop}[thm]{Proposition}
 \theoremstyle{definition}
 \newtheorem{defn}[thm]{Definition}
 \theoremstyle{remark}

 \numberwithin{equation}{section}

\newcommand{\ie}{i.e.}
\newcommand{\f}{\phi}
\newcommand{\tg}{\tilde{g}}
\newcommand{\n}{\nabla}

\newcommand{\M}{(\mathcal{M},\allowbreak{}\f,\allowbreak{}\xi,\allowbreak{}\eta,g)}

\newcommand{\R}{\mathbb R}

\newcommand{\X}{\mathfrak X}
\newcommand{\F}{\mathcal{F}}

\newcommand{\HH}{\mathcal{H}}

\newcommand{\MM}{\mathcal{M}}
\newcommand{\LL}{\mathfrak{L}}
\newcommand{\LLL}{\mathcal{L}}
\newcommand{\hatN}{\widehat{N}}
\newcommand{\tr}{{\rm tr}}
\newcommand{\ta}{\theta}
\newcommand{\ze}{\zeta}
\newcommand{\om}{\omega}
\newcommand{\lm}{\lambda}

\newcommand{\D}{\mathrm{d}}

\newcommand{\thmref}[1]{Theorem~\ref{#1}}

\newcommand{\cororref}[1]{Corollary~\ref{#1}}
\newcommand{\propref}[1]{Proposition~\ref{#1}}
\newcommand{\tablref}[1]{Table~\ref{#1}}
\title{On Almost Paracontact Almost Paracomplex Riemannian Manifolds}

\author{Mancho H. Manev and Veselina R. Tavkova}
\begin{document}

\maketitle

\begin{abstract}
 Almost paracontact manifolds of odd dimension having
an almost paracomplex structure on the paracontact distribution are studied.
The components of the fundamental (0,3)-tensor, derived by the covariant
derivative of the structure endomorphism and the metric
on the considered manifolds in each of the basic classes, are obtained.
Then, the case of the lowest dimension 3 of these manifolds is considered.
An associated tensor of the Nijenhuis tensor is introduced and the studied manifolds
are characterized with respect to this pair of tensors.
Moreover, a cases of paracontact and para-Sasakian types are commented.
A family of examples is given.

\end{abstract}

\section{Introduction}\label{sec-intro}
 \vglue-10pt
 \indent

In 1976, on a differentiable manifold of arbitrary dimension,  I. Sato introduced  in
\cite{Sato76} the concept of (almost) paracontact structure compatible with a
Riemannian metric as an analogue of almost contact Riemannian manifold. Then, he
studied several properties of the considered manifolds. Later, a lot of geometers
develop the differential geometry of these manifolds and in particular of paracontact
Riemannian manifolds and para-Sasakian manifolds. In the beginning are the papers
\cite{Sato77}, \cite{AdatMiya77}, \cite{Sato78}, \cite{SatoMats79} and \cite{Sa80}
by I. Sato, T. Adati, T. Miyazawa, K. Matsumoto and S. Sasaki. \footnote{The authors were supported in part by  project FP17-FMI-008
of the Scientific Research Fund, University of Plovdiv, Bulgaria.}

On an almost paracontact manifold can be considered two kinds of metrics compatible
with the almost paracontact structure.
If the structure endomorphism induces an isometry on the paracontact distribution of
each tangent fibre, then the manifold has an almost paracontact Riemannian structure
as in the papers mentioned above.
In the case when the induced transformation is antiisometry, then the manifold has a
structure of an almost paracontact metric manifold, where the metric is
semi-Riemannian of type $(n+1,n)$. This case is studied by many geometers, see for
example the papers \cite{NakZam}, \cite{ZamNak} of S. Zamkovoy and G. Nakova.

In 2001, M. Manev and M. Staikova give a classification in \cite{ManSta01}
of almost paracontact Riemannian manifold of type $(n,n)$
according to the notion given by Sasaki in \cite{Sa80}.
These manifolds are of dimension $2n+1$ and
the induced almost product structure on the paracontact distribution is traceless,
i.e. it is an almost paracomplex structure.

In the present paper, we continue investigations on these manifolds.
The paper is organized as follows.
In Sect.~\ref{sec-mfds}, we
recall some facts about the almost paracontact Riemannian manifolds
of the considered type and we make some additional comments.
In Sect.~\ref{sec-3dim}, we reduce the basic classes
of the considered manifolds in the case of the lowest dimension 3.
In Sect.~\ref{sec-PRM} and Sect.~\ref{sec-TNT}, we find the class of paracontact type
and the class of normal type of the manifolds studied, respectively, and we obtain
some related properties.
In Sect.~\ref{sec-ANT}, we introduce an associated Nijenhuis tensor and we discuss
relevant problems.
In Sect.~\ref{sec-FNhaN}, we argue that the classes of the considered manifolds can be
determined only by the pair of Nijenhuis tensors.
Finally, in Sect.~\ref{sec-exm}, we construct a family of Lie groups as examples of
the manifolds of the studied type and we characterize them in relation with the above
investigations.


 \section{Almost paracontact almost paracomplex Riemannian manifolds}\label{sec-mfds}
 \vglue-10pt
 \indent

Let $(\MM,\f,\xi,\eta)$ be an \emph{almost paracontact manifold}, \ie{}   $\MM$
is an $m$-dimensional real differentiable manifold with an \emph{almost
paracontact structure} $(\f,\xi,\eta)$ if it admits a tensor field
 $\f$  of type $(1,1)$ of the tangent bundle, a vector field $\xi$ and a 1-form
 $\eta$, satisfying the following conditions:
\begin{equation}\label{str}
\begin{array}{c}
\f\xi = 0,\quad \f^2 = I - \eta \otimes \xi,\quad
\eta\circ\f=0,\quad \eta(\xi)=1,
\end{array}
\end{equation}
where $I$ is the identity on the tangent bundle \cite{Sato76}.

In \cite{Sa80}, it is considered the so-called almost paracontact manifold of type
$(p, q)$, where $p$ and $q$ are the numbers of the multiplicity of
the $\f$'s eigenvalues $+1$ and $-1$, respectively. Moreover, $\f$ has a simple
eigenvalue $0$. Therefore, we have $\tr\f=p-q$.

arbitrary dimension.

Let us recall that an \emph{almost product structure} $P$ on an differentiable
manifold of arbitrary dimension $m$ is an endomorphism on the manifold such that
$P^2=I$. 
Then a manifold with such a structure is called an \emph{almost product manifold}.
In the case when the eigenvalues $+1$ and $-1$ of $P$ have one and the same
multiplicity $n$, the structure $P$ is called an \emph{almost paracomplex structure}
and the manifold is known as an \emph{almost paracomplex manifold} of dimension $2n$
\cite{CrFoGa96}. Then $\tr P=0$ follows.

Further we consider the case when the dimension of $\MM$ is $m=2n+1$. Then
$\HH=\ker(\eta)$ is the $2n$-dimensional paracontact distribution of the tangent
bundle of
$(\MM,\f,\xi,\eta)$, the endomorphism $\f$ acts as an almost paracomplex structure on
each fiber of $\HH$ and the pair $(\HH,\f)$ induces a $2n$-dimensional  almost
paracomplex manifold. Then we give the following
\begin{defn}
A $(2n+1)$-dimensional differentiable manifold with a structure $(\f,\xi,\eta)$
defined by  (\ref{str}) and $\tr\f=0$ is called \textit{almost paracontact almost
paracomplex manifold}. We denote it by  $(\MM,\f,\xi,\eta)$.
\end{defn}

direct sum
decomposed uniquely in the form
following conditions
=0.

Now we can introduce a metric on the considered manifold.
It is known from \cite{Sato76} that $\MM$ admits a Riemannian metric $g$ which is
compatible with the structure of the manifold by the following way:
\begin{equation}\label{g}
g(\f x,\f y)= g(x,y)-\eta (x)\eta (y),\quad g( x,\xi)= \eta (x).
\end{equation}
Here and further $x$, $y$, $z$ will stand for arbitrary
elements of the Lie algebra $\X(\MM)$ of tangent vector fields on $\MM$ or vectors in
the tangent space $T_p\MM$ at
$p\in \MM$.

In \cite{Sato77}, an almost paracontact manifold of arbitrary dimension with a
Riemannian metric $g$ defined by (\ref{g}) is called an \emph{almost paracontact
Riemannian manifold}.

It is easy to conclude that the requirement for a positive definiteness of the metric
is not necessary, i.e  $g$ can be a pseudo-Riemannian metric. 
Then, since $g(\xi,\xi)=1$ follows from (\ref{str}) and (\ref{g}), the signature of
$g$ has the form $(2k+1,2n-2k)$, $k<n$. Since the signature of the metric is not
crucial for our considerations, we suppose that $g$ is Riemannian.

\begin{defn}
Let the manifold $(\MM,\f,\xi,\eta)$ be equipped with a Riemannian metric $g$
satisfying (\ref{g}). Then $(\MM,\f,\xi,\eta,g)$ is called an \textit{almost
paracontact almost paracomplex Riemannian manifold}.
\end{defn}

The decomposition $x=\f^2x+\eta(x)\xi$ due to  (\ref{str}) generates the projectors
$h$ and $v$ on any tangent space of $(\MM,\f,\xi,\eta)$. These projectors are
determined by $hx=\f^2x$ and $vx=\eta(x)\xi$ and have the properties $h\circ h =h$,
$v\circ
v=v$, ${h\circ v=\allowbreak{}v\circ h=0}$. Therefore, we have the
orthogonal decomposition $T_p\MM=h(T_p\MM)\oplus v(T_p\MM)$. Obviously, it generates
the corresponding orthogonal decomposition of the space $\mathcal{S}$ of the
tensors $S$ of type (0,2) over $(\MM,\f,\xi,\eta)$.
This decomposition is invariant with respect to transformations preserving the
structures of the manifold.
Hereof, we use the following
linear operators in $\mathcal{S}$:
\begin{equation}\label{ell}
\begin{array}{l}
\ell_1(S)(x,y)=S(hx,hy),\qquad \ell_2(S)(x,y)=S(vx, vy),\\[4pt]
\ell_3(S)(x,y)=S(vx, hy)+S(hx, vy).
\end{array}
\end{equation}
Namely, we have the following decomposition:
\[
\mathcal{S}=\ell_1(\mathcal{S})\oplus \ell_2(\mathcal{S})
\oplus \ell_3(\mathcal{S}),
\qquad
\ell_i(\mathcal{S})=\left\{S\in \mathcal{S}\ |\
S=\ell_i(S)\right\},\quad i=1,2,3.
\]
The associated metric $\tg$ of $g$ on $\M$ is defined by
$\tg(x,y)=g(x,\f y)+\eta(x)\eta(y)$. It is shown that $\tg$ is a compatible metric
with $(\MM,\f,\xi,\eta)$ and it is a pseudo-Riemannian metric of signature $(n + 1,
n)$. Therefore,
$(\MM,\f,\xi,\eta,\tg)$ is also an almost paracontact almost paracomplex manifold but
with a pseudo-Riemannian metric.

Since the metrics $g$ and $\tg$ belong to $\mathcal{S}$, then
they have corresponding components in the three orthogonal subspaces introduced above
and we get them in the following form:
\[
\begin{array}{lll}
\ell_1(g)=g(\f\cdot,\f\cdot)=g-\eta\otimes\eta, \quad
&\ell_2(g)=\eta\otimes\eta,
\quad &\ell_3(g)=0,\\[0pt]
\ell_1(\tg)=g(\cdot,\f\cdot)=\tg-\eta\otimes\eta, \quad
&\ell_2(\tg)=\eta\otimes\eta, \quad &\ell_3(\tg)=0.
\end{array}
\]

In the final part of the present section we recall the needed notions and results from
\cite{ManSta01}.

In the cited paper, the manifolds under study are called almost
paracontact Riemannian manifolds of type $(n,n)$.
The structure group of $\M$ is $\mathcal{O}(n)\times\mathcal{O}(n)\times 1$, where
$\mathcal{O}(n)$ is the group of the orthogonal matrices of size $n$.


The tensor  $F$ of type (0,3) plays a fundamental role in differential geometry of the
considered manifolds. It is defined by:
\begin{equation}\label{F=nfi}
F(x,y,z)=g\bigl( \left( \nabla_x \f \right)y,z\bigr),
\end{equation}
where $\nabla $ is the Levi-Civita connection of $g$.
The basic properties of $F$ with respect to the structure are the following:
\begin{equation}\label{F-prop}
\begin{array}{l}
F(x,y,z)=F(x,z,y)\\
\phantom{F(x,y,z)}=-F(x,\f y,\f z)+\eta(y)F(x,\xi,z)
+\eta(z)F(x,y,\xi).
\end{array}
\end{equation}
The relations of $\nabla \xi$ and $\nabla \eta$ with $F$ are:
\begin{equation}\label{n_eta_F}
(\nabla_x \eta)y=g\left( \nabla_x \xi, y \right)=-F(x,\f y,\xi).
\end{equation}

If $\left\{\xi;e_i\right\}$ $(i=1,2,\dots,2n)$ is a basis of the tangent space
$T_p\MM$ at an arbitrary point $p\in \MM$ and $\left(g^{ij}\right)$ is the inverse
matrix of the
matrix $\left(g_{ij}\right)$ of $g$, then the following 1-forms
are associated with $F$:
\begin{equation}\label{t}
\theta(z)=g^{ij}F(e_i,e_j,z),\quad
\theta^*(z)=g^{ij}F(e_i,\f e_j,z), \quad \omega(z)=F(\xi,\xi,z).
\end{equation}
These 1-forms are known also as the Lee forms of the considered manifolds. Obviously,
the identities
$\om(\xi)=0$ and $\ta^*\circ\f=-\ta\circ\f^2$ are always valid.

There, it is made a classification of the almost paracontact almost paracomplex
Riemannian manifolds with respect to $F$.
The vector space $\mathbb{F}$ of all tensors $F$ with the properties (\ref{F-prop}) is
decomposed into 11 subspaces $\mathbb{F}_{i}$ $(i=1,2,\dots,11)$, which are orthogonal
and invariant with respect to the structure group of the considered manifolds. This
decomposition induces a classification of the manifolds under study.
An almost paracontact almost paracomplex Riemannian manifold is said to be in the
class $\F_{i}$ $(i=1,2,\dots,11)$, or briefly an $\F_{i}$-manifold, if the tensor $F$
belongs to the subspace $\mathbb{F}_{i}$.
Such a way, it is obtained that this classification consists of 11 basic
classes $\F_1$, $\F_2$, $\dots$, $\F_{11}$.
The intersection of the basic classes is the special class $\F_0$
determined by the condition $F(x,y,z)=0$. Hence $\F_0$ is the
class of the considered manifolds with $\n$-parallel
structures, i.e. $\n\f=\n\xi=\n\eta=\n g=\n \tg=0$.

Moreover, it is given the conditions for $F$ determining the basic classes $\F_{i}$ of
$\M$ and the components of $F$ corresponding to $\F_{i}$.
It is said that $\M$ belongs to the class $\F_{i}$ $(i=1,2,\dots,11)$
if and only if the equality $F=F_i$ is valid. In the last expression, $F_{i}$ are the
components of $F$ in the subspaces $\mathbb{F}_{i}$ and they are given by the
following equalities:

\begin{subequations}\label{F1-11}
\begin{equation}
\begin{split}
&F_{1}(x,y,z)=\frac{1}{2n}\bigl\{g(\f x,\f y)\ta (\phi ^{2}z) +g(\f x,\f z)\ta (\f^{2}y)\\
&\phantom{F_{1}(x,y,z)=\frac{1}{2n}\bigl\{}
-g(x,\f y)\ta (\f z)-g(x,\f z)\ta(\f y)\bigr\}, \\[0pt]
&F_{2}(x,y,z)=\frac{1}{4}\bigl\{ %
2F(\f^2 x,\f^2 y,\f^2 z)+F(\f^2 y,\f^2 z,\f^2 x)+F(\f^2 z,\f^2 x,\f^2 y)\\
&\phantom{F_{2}(x,y,z)=\frac{1}{4}\bigl\{ %
2F(\f^2 x,\f^2 y,\f^2 z)}
-F(\f y,\f z,\f^2 x) 
-F(\f z,\f y,\f^2 x)\bigr\}\\[0pt]
&\phantom{F_{2}(x,y,z)=}
-\frac{1}{2n}\bigl\{g(\f x,\f y)\ta(\f^2 z) +g(\f x,\f z)\ta(\f^2 y)\\
&\phantom{F_{1}(x,y,z)=\frac{1}{2n}\bigl\{-}
-g(x,\f y)\ta(\f z) -g(x,\f z)\ta(\f y)\bigr\}, \\[0pt]
&F_{3}(x,y,z)=\frac{1}{4}\bigl\{%
2F(\f^2 x,\f^2 y,\f^2 z)-F(\f^2 y,\f^2 z,\f^2 x) -F(\f^2  z,\f^2 x,\f^2 y)\\[0pt]
&\phantom{F_{3}(x,y,z)=\frac{1}{4}\bigl\{2F(\f^2 x,\f^2 y,\f^2 z)}%
 +F(\f y,\f z,\f ^2 x)+F(\f z,\f y,\f ^2 x)\bigr\},
\\[0pt]
&F_{4}(x,y,z)=\frac{\ta(\xi)}{2n}\bigl\{g(\f x,\f y)\eta(z)+g(\f x,\f z)\eta(y)\bigr\}, \\[0pt]
& F_{5}(x,y,z)=\frac{\ta^*(\xi)}{2n}\bigl\{g(x,\f
y)\eta(z)+g(x,\f
z)\eta(y)\bigr\},\\[0pt]
&F_{6}(x,y,z)=\frac{1}{4}\bigl\{ %
[F(\f^2 x,\f^2 y,\xi)+F(\f^2 y,\f^2 x,\xi)+F(\f x,\f y,\xi)\\[0pt]
&\phantom{F_{6}(x,y,z)=\frac{1}{4}\bigl\{ %
[F(\f^2 x,\f^2 y,\xi)+F(\f^2 y,\f^2 x,\xi)}
+F(\f y,\f x,\xi)] \eta(z)\\[0pt]
\end{split}
\end{equation}
\begin{equation}
\begin{split}
&\phantom{F_{6}(x,y,z)=-.}%
+[F(\f^2 x,\f^2 z,\xi)+F(\f^2 z,\f^2 x,\xi)+F(\f x,\f z,\xi)\\[0pt]
&\phantom{F_{6}(x,y,z)=\frac{1}{4}\bigl\{+[F(\f^2 x,\f^2 z,\xi)+F(\f^2 z,\f^2 x,\xi)\,\,}
+F(\f z,\f x,\xi)]\eta(y)\bigr\}\\[0pt]
&\phantom{F_{6}(x,y,z)=}%
-\frac{\ta(\xi)}{2n}\bigl\{g(\f x,\f y)\eta(z)+g(\f
x,\f z)\eta(y)\bigr\}\\[0pt]
&\phantom{F_{6}(x,y,z)=}
-\frac{\ta^*(\xi)}{2n}\bigl\{g(x,\f
y)\eta(z)+g(x,\f
z)\eta(y)\bigr\},\\[0pt]
&F_{7}(x,y,z)=\frac{1}{4}\bigl\{ %
[F(\f^2 x,\f^2 y,\xi)-F(\f^2 y,\f^2 x,\xi)+F(\f x,\f y,\xi)\\[0pt]
&\phantom{F_{7}(x,y,z)=\frac{1}{4}\bigl\{ %
[F(\f^2 x,\f^2 y,\xi)-F(\f^2 y,\f^2 x,\xi)}
-F(\f y,\f x,\xi)]\eta(z)\\[0pt]
&\phantom{F_{7}(x,y,z)=-.} %
+ [F(\f^2 x,\f^2 z,\xi)-F(\f^2 z,\f^2 x,\xi)+F(\f x,\f z,\xi)\\[0pt]
&\phantom{\frac{1}{4}\bigl\{\,\,[F(\f^2 x,\f^2 z,\xi)-F(\f^2 z,\f^2 x,\xi)+F(\f x,\f z,\xi)}
-F(\f z,\f x,\xi)]\eta(y)\bigr\},\\[0pt]
&F_{8}(x,y,z)=\frac{1}{4}\bigl\{ %
[F(\f^2 x,\f^2 y,\xi)+F(\f^2 y,\f^2 x,\xi)-F(\f x,\f y,\xi)\\[0pt]
&\phantom{F_{8}(x,y,z)=\frac{1}{4}\bigl\{ %
[F(\f^2 x,\f^2 y,\xi)-F(\f^2 y,\f^2 x,\xi)}
-F(\f y,\f x,\xi)]\eta(z)\\[0pt]
&\phantom{F_{8}(x,y,z)=-.} %
+ [F(\f^2 x,\f^2 z,\xi)+F(\f^2 z,\f^2 x,\xi)-F(\f x,\f z,\xi)\\[0pt]
&\phantom{F_{8}(x,y,z)=\frac{1}{4}\bigl\{\,\,+[F(\f^2 x,\f^2 z,\xi)-F(\f^2 z,\f^2 x,\xi)}
-F(\f z,\f x,\xi)]\eta(y)\bigr\},\\[0pt]
&F_{9}(x,y,z)=\frac{1}{4}\bigl\{ %
[F(\f^2 x,\f^2 y,\xi)-F(\f^2 y,\f^2 x,\xi)-F(\f x,\f y,\xi)\\[0pt]
&\phantom{F_{9}(x,y,z)=\frac{1}{4}\bigl\{ %
[F(\f^2 x,\f^2 y,\xi)-F(\f^2 y,\f^2 x,\xi)}
+F(\f y,\f x,\xi)]\eta(z)\\[0pt]
&\phantom{F_{9}(x,y,z)=-.} %
+ [F(\f^2 x,\f^2 z,\xi)-F(\f^2 z,\f^2 x,\xi)-F(\f x,\f z,\xi)\\[0pt]
&\phantom{F_{9}(x,y,z)=\frac{1}{4}\bigl\{\,\,+[F(\f^2 x,\f^2 z,\xi)-F(\f^2 z,\f^2 x,\xi)}
+F(\f z,\f x,\xi)]\eta(y)\bigr\},\\[0pt]
&F_{10}(x,y,z)= \eta (x) F(\xi, \f^2 y,\f^2 z),\\[4pt]
&F_{11}(x,y,z)= \eta (x) \bigl\{\eta(y)\om (z) + \eta(z)\om (y)\bigr\}.\\[4pt]
\end{split}
\end{equation}
\end{subequations}

It is easy to conclude that a manifold of the considered type
belongs to a direct sum of two or more basic classes, i.e.
$\M\in\F_i\oplus\F_j\oplus\cdots$, if and only if the fundamental
tensor $F$ on $\M$ is the sum of the corresponding components
$F_i$, $F_j$, $\ldots$ of $F$, i.e. the following condition is
satisfied $F=F_i+F_j+\cdots$.

Finally in this section, we obtain immediately
\begin{prop}\label{prop-dim}
The dimensions of the subspaces $\mathbb{F}_i$ $(i=1,2,\ldots ,11)$ in the
decomposition of the space $\mathbb{F}$ of the tensors $F$ on $\M$ are the following:
\[
\begin{array}{lll}
\dim\mathbb{F}_1=2n, \quad 			& \dim\mathbb{F}_2=n(n-1)(n+2), \quad 	&
\dim\mathbb{F}_3=n^2(n-1),\\[0pt]
\dim\mathbb{F}_4=1, \quad 				& \dim\mathbb{F}_5=1, \quad   							
& \dim\mathbb{F}_6=(n-1)(n+2), \\[0pt]
\dim\mathbb{F}_7=n(n-1), \quad & \dim\mathbb{F}_8=n^2, \quad 						&
\dim\mathbb{F}_9=n^2,  \\[0pt]
\dim\mathbb{F}_{10}=n^2, \quad & \dim\mathbb{F}_{11}=2n.
\end{array}
\]
\end{prop}

\begin{proof}
Using the characteristic symmetries of the fundamental tensor $F$ and the form of its
components from (\ref{F1-11}) in each of ${\F}_i$ $(i=1,2,\ldots ,11)$, we get the
equalities in the statement.
\end{proof}

 \section{The components of the fundamental tensor for dimension 3}\label{sec-3dim}
 \vglue-10pt
 \indent

Let  $(\MM,\f,\xi,\eta,g)$  be the manifold under study with the lowest dimension,
i.e. $\dim{\MM}=3$ (or $n=1$) and let the system of three vectors
$\{e_0=\xi,e_1=e,e_2=\f e\}$ be a $\f$-basis which satisfies the following
conditions:
\begin{equation}\label{gij}
\begin{array}{l}
g(e_0,e_0)=g(e_1, e_1)=g(e_2, e_2)=1,\quad \\
g(e_0, e_1)=g(e_1,e_2)=g(e_0,e_2)=0.
\end{array}
\end{equation}
We denote the components of the tensors $F$, $\ta$, $\ta^*$ and $\om$ with respect to
the $\f$-basis $\left\{e_0,e_1,e_2\right\}$ as follows ${F_{ijk}=F(e_i,e_j,e_k)}$,
${\ta_k=\ta(e_k)}$, ${\ta^*_k=\ta^*(e_k)}$ and ${\om_k=\om(e_k)}$.
The properties (\ref{F-prop}) and (\ref{gij}) imply the equalities $F_{i12}=F_{i21}=0$
and $F_{i11}=-F_{i22}$ for any $i$. Then, bearing in mind (\ref{t}), we obtain for the
Lee forms the following:
\begin{equation}\label{t3}
\begin{array}{lll}
\ta_0=F_{110}+F_{220},\qquad & \ta_1=F_{111}=-F_{122}=-\ta^*_2,\qquad
&\om_1=F_{001},\\[0pt]
\ta^*_0=F_{120}+F_{210}, \qquad &\ta_2=F_{222}=-F_{211}=-\ta^*_1,\qquad
&\om_2=F_{002},\\[0pt]
& & \om_0=0.
\end{array}
\end{equation}
The arbitrary vectors $x$, $y$, $z$ in $T_p\MM$, $p\in \MM$, have the expression
$x=x^ie_i$, $y=y^ie_i$, $z=z^ie_i$ with respect to $\left\{e_0,e_1,e_2\right\}$.
\begin{prop}\label{prop-Fi}
The components $F_i$ $(i=1,2,\dots,11)$ of the fundamental tensor
$F$ for a 3-dimensional almost paracontact almost paracomplex Riemannian manifold are
the following:
\begin{equation}\label{Fi3}
\begin{array}{l}
F_{1}(x,y,z)=\left(x^1\ta_1-x^2\ta_2\right)\left(y^1z^1-y^2z^2\right), \\[0pt]
F_{2}(x,y,z)=F_{3}(x,y,z)=0,
\\[0pt]
F_{4}(x,y,z)=\frac{\ta_0}{2}\Bigl\{x^1\left(y^0z^1+y^1z^0\right)
+x^2\left(y^0z^2+y^2z^0\right)\bigr\},\\[0pt]
F_{5}(x,y,z)=\frac{\ta^*_0}{2}\bigl\{x^1\left(y^0z^2+y^2z^0\right)
+x^2\left(y^0z^1+y^1z^0\right)\bigr\},\\[0pt]
F_{6}(x,y,z)=F_{7}(x,y,z)=0,\\[0pt]
F_{8}(x,y,z)=\lm\bigl\{x^1\left(y^0z^1+y^1z^0\right)
-x^2\left(y^0z^2+y^2z^0\right)\bigr\}
,\\[0pt]
F_{9}(x,y,z)=\mu\bigl\{x^1\left(y^0z^2+y^2z^0\right)
-x^2\left(y^0z^1+y^1z^0\right)\bigr\}
,\\[0pt]
F_{10}(x,y,z)=\nu\, x^0\left(y^1z^1-y^2z^2\right)
,\\[0pt]
F_{11}(x,y,z)=x^0\bigl\{\om_{1}\left(y^0z^1+y^1z^0\right)
+\om_{2}\left(y^0z^2+y^2z^0\right)\bigr\},
\end{array}
\end{equation}
where
\[
\lm=F_{110}=-F_{220}, \qquad
\mu=F_{120}=-F_{210}, \qquad
\nu=F_{011}=-F_{022}
\]
and the components of the Lee forms are given in (\ref{t3}).
\end{prop}
\begin{proof}
Using the expressions (\ref{F1-11}) of $F_i$ for
the corresponding classes $\F_i$ $(i=1,\allowbreak{}\dots,$ $11)$, the equalities
(\ref{F-prop}), (\ref{gij}) and (\ref{t3}), we obtain the corresponding form of $F_i$
for the lowest dimension of the considered manifold.
\end{proof}

As a result of \propref{prop-Fi}, we establish the truthfulness of the following
\begin{thm}\label{thm-3D}
The 3-dimensional almost paracontact almost paracomplex Riemannian manifolds belong to
the basic  classes
$\F_1$, $ \F_4$, $\F_5$,  $\F_8$, $\F_9$, $\F_{10}$,  $\F_{11}$ and to their direct
sums.
\end{thm}

Let us remark that for the considered manifolds of dimension 3, the basic classes
$\F_2$, $\F_3$,
$\F_6$, $\F_7$ are restricted to the special class $\F_0$.


\section{Paracontact almost paracomplex Riemannian manifolds}\label{sec-PRM}

 Let $\M$, $\dim\MM=2n+1$, be an almost paracontact almost paracomplex Riemannian
 manifold such that the following condition is satisfied:
\begin{equation}\label{PG}
2g(x,\f y)=(\LL_\xi g)(x,y),
\end{equation}
where the Lie derivative $\LL$
of $g$ along $\xi$ has the following form in terms of $\n\eta$:
\begin{equation}\label{L}
(\LL_\xi g)(x,y)=(\nabla _x\eta)y+(\nabla _y\eta)x.
\end{equation}
Bearing in mind (\ref{n_eta_F}) and (\ref{L}), $\LL_\xi g$ is expressed by $F$ as
follows:
\begin{equation}\label{LF}
(\LL_\xi g)(x,y)=-F(x,\f y,\xi )-F(y,\f x,\xi ).
\end{equation}

In \cite{Sato77}, it is said that an $m$-dimensional almost paracontact Riemannian
manifold endowed with the property $2g(x,\f y)=(\nabla _x\eta)y+(\nabla _y\eta)x$ is a
\emph{paracontact Riemannian manifold}.
\begin{defn}
An almost paracontact almost paracomplex Riemannian manifold satisfied (\ref{PG}) is
called  \textit{paracontact almost paracomplex Riemannian manifold}.
\end{defn}
Now we determine the class of paracontact almost paracomplex Riemannian manifolds with
respect to the basic classes $\F_i$. 
Firstly, we compute $\LL_\xi g$ on each $\F_i$-manifold
using (\ref{LF}) and (\ref{F1-11}). Then we obtain

\begin{prop}
Let $\M$ be an almost paracontact almost paracomplex Riemannian manifold. Then we
have:
\begin{itemize}
\item[a)]  $(\LL_\xi g)(x,y) =0$ if and only if $\M$ belongs to
    $\F_1\oplus\F_2\oplus\F_3\oplus\F_7\oplus\F_8\oplus\F_{10}$;
\item[b)]  $(\LL_\xi g)(x,y) =-\frac{1}{n}\theta(\xi)g(x,\f y)$ if and only if $\M$
    belongs to $\F_4$;
\item[c)]  $(\LL_\xi g)(x,y) =-\frac{1}{n}\theta^*(\xi)g(\f x,\f y)$ if and only if
    $\M$ belongs to $\F_5$;
\item[d)]  $(\LL_\xi g)(x,y) =2\bigl(\nabla _x\eta \bigr)y$ if and only if $\M$
    belongs to $\F_6\oplus\F_9$;
\item[e)]  $(\LL_\xi g)(x,y) =-\eta (x)\om (\f y) - \eta (y)\om (\f x)$ if and only
    if $\M$ belongs to $\F_{11}$.
\end{itemize}
\end{prop}
It is known that $\xi$ is a Killing vector field when $\LL_\xi g = 0$. Therefore, the
latter proposition implies
\begin{cor}\label{cor:Kill}
An almost paracontact almost paracomplex Riemannian manifold $\M$ has a Killing vector
field $\xi$ if and only if $\M$ belongs to $\F_i$ $(i=1,2,3,7,8,10)$ or to their
direct sums.
\end{cor}

We denote by ${\F_4}'$ the subclass of $\F_4$ determined by $\ta(\xi)=-2n$, i.e.
\begin{equation}\label{F4'}
{\F_4}'=\bigl\{\F_4\, |\, \ta(\xi)=-2n\bigr\}.
\end{equation}
Then, the component ${F_4}'$ of $F$ corresponding to the subclass ${\F_4}'$ is
\begin{equation}\label{F4'=}
	{F_4}'(x,y,z)=-g(\f x,\f y)\eta(z)-g(\f x,\f z)\eta(y).
\end{equation}
\begin{thm}\label{thm:para}
Paracontact  almost paracomplex Riemannian manifolds belong to ${\F_4}'$ or to its
direct sums with $\F_1$, $\F_2$, $\F_3$, $\F_7$, $\F_8$ and $\F_{10}$.
\end{thm}
\begin{proof}
Let us consider an arbitrary almost paracontact almost paracomplex Riemannian
manifold, i.e. $F=F_1+\dots+F_{11}$.
Using the expressions  (\ref{F1-11}) of $F_i$ for the corresponding classes $\F_i$ $(i
= 1,\dots, 11)$ and the condition (\ref{PG}), we obtain
\begin{equation}\label{Fpara}
F=F_1+F_2+F_3+{F_4}'+F_7+F_8+F_{10},
\end{equation}
where ${F_4}'$ is determined by (\ref{F4'=}).

Vice versa, if (\ref{Fpara}) holds true, then it implies (\ref{PG}) by (\ref{LF}),
i.e. $\M$ is a paracontact almost paracomplex Riemannian manifold.
Supposing that $\M$ belongs to some of $\F_i$ $(i = 1,2,3,7,8,10)$ or their direct
sum, it follows that $g$ is degenerate.
Therefore, the component ${F_4}'$ is indispensable and  we get the statement.
\end{proof}

Let us remark that ${\F_4}'$ and ${\F_0}$ are subclasses of ${\F_4}$ without common
elements.

Moreover, bearing in mind \cororref{cor:Kill} and \thmref{thm:para}, we conclude that
paracontact almost paracomplex Riemannian manifolds with a Killing vector field $\xi$
do not exist, i.e. for the manifolds studied, there is no analogue of a K-contact
manifold.


In \cite{SatoMats79}, it is introduced the notion of a para-Sasakian Riemannian
manifold of an arbitrary dimension by the condition $\f x=\n_x \xi$.
The same condition determines a special kind of paracontact almost paracomplex
Riemannian manifolds. These manifolds we call para-Sasakian paracomplex Riemannian
manifolds. Then, using (\ref{F4'}), we obtain the truthfulness of the following
\begin{thm}\label{thm:pS}
The class of the para-Sasakian paracomplex Riemannian manifolds is ${\F_4}'$.
\end{thm}

\section{The Nijenhuis tensor}\label{sec-TNT}

\subsection{Introduction of the Nijenhuis tensor}

Let us consider the product manifold $\check{\MM}$ of an almost paracontact almost
paracomplex manifold $(\MM,\phi,\allowbreak{}\xi,\allowbreak{}\eta)$ and the real line
$\R$, i.e. $\check\MM = \MM \times\R$. We denote a vector field on $\check\MM$ by
$\left(x,a\frac{\D}{\D r} \right)$, where $x$ is tangent to $\check\MM$, $r$ is the
coordinate on $\R$ and $a$  is a function on  $\MM \times\R$. Further, we use the
denotation $\partial_r=\frac{\D}{\D r}$ for brevity. Following \cite{Sato76}, we
define an almost paracomplex structure $\check P$ on $\check \MM$ by:
\begin{equation}\label{P}
\check P \left(x,\,a\partial_r \right) =\left(\f x +\frac{a}{r}\xi,\,
r\eta(x)\partial_r \right)
\end{equation}
that implies
\[
\check P x=\f x,\qquad \check P\xi=r\partial_r,\qquad \check
P\partial_r=\frac{1}{r}\xi.
\]
Further, we use the setting $\ze=r\partial_r$.
It easy to check that ${\check{P}}^2 = I$ and $\tr \check P=0$. In the case when
$\check P$ is integrable, it is said that the almost paracontact structure
$(\f,\xi,\eta)$ is {\em normal}.

It is known, the vanishing of the Nijenhuis torsion $\left[\check P,\check P\right]$
of $\check P$ is a necessary and sufficient condition for integrability of $\check
P$.
According to \cite{Sato76}, the condition of normality is equivalent to vanishing of
the following four tensors:
\begin{equation}\label{NK}
\begin{array}{l}
N^{(1)}(x,y)=[\f,\f](x,y) - \D\eta(x,y)\xi,\\[0pt]
N^{(2)}(x,y)=(\LL_{\f x} \eta)(y) - (\LL_{\f y} \eta)(x),\\[0pt]
N^{(3)}(x)=(\LL_\xi \f)(x),\\[0pt]
N^{(4)}(x)=(\LL_\xi \eta)(x),
\end{array}
\end{equation}
where the Nijenhuis torsion of $\f$ is determined by:
\begin{equation}\label{[ff]}
[\f,\f](x,y)=[\f x,\f y]+\f^2[x,y]-\f[\f x,y]-\f[x,\f y]
\end{equation}
and $\D\eta$ is the exterior derivative of $\eta$ given by:
\begin{equation}\label{deta}
\D\eta (x,y)=(\nabla _x\eta)y-(\nabla _y\eta)x.
\end{equation}

According to (\ref{n_eta_F}) and (\ref{deta}), $\D\eta$ is expressed by $F$ as
follows:
\begin{equation}\label{detaF}
\D\eta (x,y)=-F(x,\f y,\xi )+F(y,\f x,\xi ).
\end{equation}

Let $\M$  be a $(2n+1)$-dimensional almost paracontact almost paracomplex Riemannian
manifold.

In \cite{Sato76}, it is proved that the vanishing of $N^{(1)}$ implies the vanishing
of $ N^{(2)}$, $ N^{(3)}$, $ N^{(4)}$. Then $N^{(1)}$ is denoted simply by $N$, i.e.
\begin{equation}\label{Ndef}
N(x,y) = [\f,\f](x,y) - \D\eta(x,y) \xi,
\end{equation}
and it is called the \emph{Nijenhuis tensor of the structure $(\f,\xi,\eta)$}.
Therefore, an almost paracontact structure $(\f,\xi,\eta)$  is normal if and only if
its Nijenhuis tensor is zero.

Obviously, $N$ is an antisymmetric tensor, i.e. $N(x,y) = -N(y,x)$. According to
(\ref{[ff]}), (\ref{deta}) and (\ref{Ndef}),  the tensor $N$ has the following form in
terms of the covariant derivatives of $\f$ and $\eta$ with respect to $\nabla$:
\[
N(x,y)=(\nabla_{\f x} \f)y- (\nabla_{\f y} \f)x-\f(\nabla_ x \f)y+\f(\nabla_ y
\f)x-(\nabla_x \eta)y\,\xi+(\nabla_y \eta)x\,\xi.
\]

The corresponding tensor of type (0,3) of the Nijenhuis tensor on $\M$ is defined by
equality  $N(x,y,z)=g\left(N(x,y),z\right)$.
Then, using (\ref{F=nfi}) and (\ref{n_eta_F}),  we express $N$ in terms of the
fundamental  tensor $F$ as follows:
\begin{equation}\label{N}
\begin{array}{ll}
N(x,y,z)=F(\f x,y,z)-F(\f y,x,z)-F(x,y,\f z)+F(y,x,\f z)\\[0pt]
\phantom{N(x,y,z)=}+\eta(z)\left\{F(x,\f y,\xi)-F(y,\f
x,\xi)\right\}.
\end{array}
\end{equation}

\begin{prop}\label{prop-p2}
The Nijenhuis tensor on an almost paracontact almost paracomplex Riemannian manifold
has the following properties:
\[
\begin{array}{ll}
N(\f^2 x, \f y,\f z)= - N(\f^2 x, \f^2 y,\f^2 z),  \qquad& N(\f^2 x, \f^2 y,\f^2 z) =
N(\f x, \f y,\f^2 z), \\[0pt]
N( x, \f^2 y,\f^2 z)= - N(x, \f y,\f z), \qquad& N(\f^2 x, \f^2 y, z) = N(\f x, \f y,
z), \\[0pt]
N(\xi, \f y,\f z) = {-} N(\xi, \f^2 y,\f^2 z), \qquad& N(\f x, \f y, \xi) = N(\f^2 x,
\f^2 y, \xi). \\[0pt]
\end{array}
\]
\end{prop}
\begin{proof}
The equalities from the above follow by direct computations from the properties
(\ref{F-prop}) and the expression (\ref{N}).
\end{proof}

In \cite{Sato76}, there are given the following relations between the tensors
$N^{(1)}$, $N^{(2)}$, $ N^{(3)}$ and $ N^{(4)}$:
\begin{equation}\label{NkN}
	\begin{array}{l}
	N^{(2)}(x,y) = -\eta \left( N^{(1)}(x,\f y)\right)-\eta \left( N^{(1)}(\f
x,\xi)\right)\eta(y),\\[0pt]
	N^{(3)}(x) = - N^{(1)}(\f x,\xi),\\[0pt]
	N^{(4)}(x) = - N^{(2)}(\f x,\xi),\qquad
	N^{(4)}(x) = -\eta \left( N^{(3)}(\f x)\right).
\end{array}
\end{equation}

Applying the expression (\ref{N}) to equalities (\ref{NkN}), we obtain the form of
$N^{(2)}$, $ N^{(3)}$ and $ N^{(4)}$ in terms of the fundamental tensor $F$:
\begin{equation}\label{NkF}
	\begin{array}{l}
 N^{(2)}(x,y) = - F(x,y,\xi) + F(y,x,\xi) - F(\f x,\f y,\xi) + F(\f y,\f
 x,\xi),\\[0pt]
 N^{(3)}(x, y)= F(\xi,x,y) - F(x,y,\xi) + F(\f x,\f y,\xi),\\[0pt]
 N^{(4)}(x) = - F(\xi,\xi,\f x),
\end{array}
\end{equation}
where it is used the denotation $N^{(3)}(x, y) = g\left(N^{(3)}(x), y\right)$.

\begin{prop}\label{prop-p3}
Let $(\MM,\f,\xi,\eta,g)$ be an $\F_{i}$-manifold $(i=1,2,\dots,11)$. Then the four
tensors $N^{(k)}$ ${(k=1,2,3,4)}$ on this manifold have the form in the respective
cases, given in \tablref{tab:NijenhuisTensors}.
\begin{table}
 \caption{Nijenhuis tensors}\label{tab:NijenhuisTensors}
  \vglue4mm
\footnotesize{
\begin{tabular}{|c||c|c|c|c|}
\hline
& ${N^{(1)}(x,y,z)}$ & $N^{(2)}(x,y)$ & $N^{(3)}(x,y)$ & $N^{(4)}(x)$ \\
\hline
\hline
$\F_{1}$ & 0 & {0} & {0} & {0}\\
$\F_{2}$ & {0}& {0} & {0} & {0}\\
$\F_{3}$ & ${-2\bigl\{F(\f x,\f y,\f z)
+F(\f^2 x,\f^2 y,\f z)\bigr\}}$& {0} & {0} & {0}\\
$\F_{4}$ & {0}& {0} & {0} & {0}\\
$\F_{5}$ & {0}& {0} & {0} & {0}\\
$\F_{6}$ & {0}& {0} & {0} & {0}\\
$\F_{7}$ & $4F(x,\f y,\xi)\eta(z)$ & $-4F(x,y,\xi)$ & {0} & {0}\\
$\F_{8}$ & $2\bigl\{\eta(x)F(y,\f z,\xi)
-\eta(y)F(x,\f z,\xi)\bigr\}$& {0} & $-2F(x,y,\xi)$ & {0}\\
$\F_{9}$ & $2\bigl\{\eta(x)F(y,\f z,\xi) -\eta(y)F(x,\f
z,\xi)\bigr\}$& {0} & $-2F(x,y,\xi)$ & {0}\\
$\F_{10}$ & $-\eta(x)F(\xi,y,\f z)+\eta(y)F(\xi,x,\f z)$ & {0}& $F(\xi,x,y)$ & {0} \\
$\F_{11}$ & $\eta(z)\bigl\{\eta(x)\om(\f y)-\eta(y)\om(\f
x)\bigr\}$ & $\eta(y)\om(x)-\eta(x)\om(y)$ & $\eta(y)\om(x)$ & $-\om(\f x)$ \\
\hline
\end{tabular}
}
\end{table}
\end{prop}
\begin{proof}
We apply direct computations, using  (\ref{F1-11}), (\ref {N}) and (\ref {NkF}).
\end{proof}

By virtue \propref{prop-p3}, we have the following
\begin{thm}\label{thm-F123}
An almost paracontact almost paracomplex Riemannian manifold $\M$ has:
\begin{enumerate}
  \item[a)]
   vanishing $N^{(1)}$ if and only if it belongs to some of the basic classes
   $\F_1$, $\F_2$, $\F_4$, $\F_5$,  $\F_6$ or to their direct sums;
  \item[b)]
  vanishing $N^{(2)}$ if and only if it belongs to some of the basic classes
  $\F_1,\dots, \F_6$, $\F_8$, $\F_9$,  $\F_{10}$ or to their direct sums;
  \item[c)]
  vanishing $N^{(3)}$ if and only if it belongs to the basic classes $\F_1,\dots,
  \F_7$ or to some of their direct sums;
    \item[d)]
  vanishing $N^{(4)}$ if and only if it belongs to some of the basic classes
  $\F_1,\dots, \F_{10}$ or to their direct sums.
\end{enumerate}
\end{thm}
Bearing in mind \thmref{thm-F123}, we conclude the following
\begin{cor}
The class of normal almost paracontact almost paracomplex Riemannian manifolds is
$\F_1\oplus\F_2\oplus\F_4\oplus\F_5\oplus\F_6$.
\end{cor}

\subsection{The exterior derivative of the structure 1-form}

According to (\ref{ell}), the 2-form $\D\eta$ on $\M$ can be decomposed as follows:
\[
\D\eta=\ell_1(\D\eta)+\ell_3(\D\eta),
\]
\begin{equation}\label{ell123}
\begin{array}{ll}
\ell_1(\D\eta)(x,y)=\D\eta(hx,hy),\qquad \ell_2(\D\eta)(x,y)=0,\\[0pt]
\ell_3(\D\eta)(x,y)=\D\eta(vx,hy)+\D\eta(hx,vy).
\end{array}
\end{equation}

The next proposition gives geometric conditions for vanishing the components of
$\D\eta$.
\begin{prop}\label{DT}
Let $\M$ be an almost paracontact almost paracomplex Riemannian manifold. Then we
have:
\begin{itemize}
\item[a)] the paracontact distribution $\HH$ of $\M$ is involutive if and only if
    $\ell_1(\D\eta) = 0$;
\item[b)] the integral curves of $\xi$ are geodesics on $\M$ if and only if
    $\ell_3(\D\eta)=0$.
\end{itemize}
\end{prop}
\begin{proof}
It is said that $\HH$ is an involutive distribution when $[x,y]$ belongs to $\HH$ for
$x,y\in\HH$, i.e.
$\eta\left([hx,hy]\right)=0$ holds for arbitrary $x$ and $y$. By virtue of the
identity
$\eta\left([hx,hy]\right)=-\D\eta(hx,hy)$ and (\ref{ell123}), we have the equality
$\eta\left([hx,hy]\right)=-\ell_1(\D\eta)(x,y)$. This accomplishes the proof of a).

As it is known, the integral curves of $\xi$ are geodesics on $\M$ if and only if
$\n_\xi \xi$ vanishes.
The equality (\ref{ell123}) implies that $\ell_3(\D\eta)=0$ is valid if and only if
$\D\eta(x,\xi)=0$ holds.
Applying (\ref{deta}) and (\ref{n_eta_F}), we obtain the equality
$\D\eta(x,\xi)=-g\left(\n_\xi \xi,x\right)$.  Then, it is clear that b) holds true.
\end{proof}

Next, we compute $\D\eta$  on the considered manifold belonging to each of the basic
classes and obtain the following
\begin{prop}\label{prop-eta}
Let $\M$ be an almost paracontact almost paracomplex Riemannian manifold. Then we
have:
\begin{itemize}
\item[a)]  $\D\eta(x,y) =0$ if and only if $\M$ belongs to $\F_i$ ($i=1, \ldots, 6,
    9, 10$) or to their direct sums;
\item[b)]  $\D\eta(x,y) =\ell_1(\D\eta)(x,y)=2\bigl(\nabla_x\eta \bigr)y$ if and
    only if $\M$ belongs to $\F_7$, $\F_8$ or $\F_7\oplus\F_8$;
\item[c)]  $\D\eta(x,y) =\ell_3(\D\eta)(x,y)=-\eta (x)\om (\f y) + \eta (y)\om (\f
    x)$ if and only if $\M$ belongs to $\F_{11}$.
\end{itemize}
\end{prop}
By \propref{DT} and \propref{prop-eta}, we get the following theorem, which gives a
geometric characteristic of the manifolds of some classes with respect to the form of
$\D\eta$.
\begin{thm}\label{prop-F123}
Let $\M$ be an almost paracontact almost paracomplex Riemannian manifold. Then we
have:
\begin{enumerate}
\item[a)] the structure 1-form $\eta$ is closed if and only if $\M$ belongs to
    $\F_i$ ($i=1, \dots,6,9,10$) or to their direct sums;
\item[b)] the paracontact distribution $\HH$ of $\M$ is involutive if and only if
    $\M$ belongs to $\F_i$  ($i=1, \dots,6,9,10,11$) or to their direct sums;
\item[c)] the integral curves of the structure vector field $\xi$ are geodesics on
    $\M$ if and only if $\M$ belongs to  $\F_i$ ($i=1, \dots,10$) or to their direct
    sums.
\end{enumerate}
\end{thm}

\subsection{The Nijenhuis torsion of the structure endomorphism of the paracontact
distribution}

\begin{prop}\label{prop:[fifi]}
Let $\M$ be an almost paracontact almost paracomplex Riemannian manifold. Then for the
Nijenhuis torsion of $\f$ we have:
\begin{itemize}
\item[a)]  $[\f,\f](x,y) =0$ if and only if $\M$ belongs to $\F_i$ ($i=1, 2, 4, 5,
    6, 11$) or to their direct sums;
\item[b)]  $[\f,\f](x,y)=-2\left\{\f\left(\n_{\f x} \f \right)\f y+\f\left(\n_{\f^2
    x} \f \right)\f^2 y\right\}$
if and only if $\M$ belongs to $\F_3$;
\item[c)]  $[\f,\f](x,y)=-2\left(\n_x \eta\right)(y)\,\xi$ if and only if $\M$
    belongs to $\F_7$;
\item[d)]  $[\f,\f](x,y)=-2\left\{\eta(x)\n_y \xi - \eta(y) \n_x \xi -\left(\n_x
    \eta\right)(y)\,\xi\right\}$
if and only if $\M$ belongs to $\F_8$;
\item[e)]  $[\f,\f](x,y)=-2\left\{\eta(x)\n_y \xi - \eta(y) \n_x \xi\right\}$
if and only if $\M$ belongs to $\F_9$;
\item[f)]  $[\f,\f](x,y)=-\eta(x)\,\f\left(\n_\xi \f
    \right)y+\eta(y)\,\f\left(\n_\xi \f \right)x$ if and only if $\M$ belongs to
    $\F_{10}$.
\end{itemize}
\end{prop}
\begin{proof}
Using (\ref{Ndef}) and the forms of the Nijenhuis tensor $N$ and the 2-form $\D\eta$,
given in \propref{prop-p3} and \propref{prop-eta}, respectively, we get the statements
from the above by direct computations.
\end{proof}

Now, we specialize the form of $[\f,\f]$ for the class of paracontact almost
paracomplex Riemannian manifolds and we find its subclasses of manifolds whose almost
paracomplex structure $\f$ on $\HH$ is integrable.
\begin{thm}\label{thm-[fifi]}
Let $\M$ be a paracontact almost paracomplex Riemannian manifold. Then it has:
\begin{enumerate}
\item[a)] an integrable almost paracomplex structure $\f$, i.e. $[\f,\f]=0$, if and
    only if the manifold belongs to ${\F_4}'$ or to its direct sums with $\F_1$ and
    $\F_2$;
\item[b)] an nonintegrable almost paracomplex structure $\f$, i.e. $[\f,\f]\neq 0$
    if and only if the manifold belongs to the rest of the classes, given in
    \thmref{thm:para}.
\end{enumerate}
\end{thm}
\begin{proof}
We establish the truthfulness of the statements using \thmref{thm:para} and
\propref{prop:[fifi]}.
\end{proof}

Bearing in mind \thmref{thm-[fifi]} a), the manifolds from the classes ${\F_4}'$,
$\F_1\oplus{\F_4}'$, $\F_2\oplus{\F_4}'$  and $\F_1\oplus\F_2\oplus{\F_4}'$ we call
\emph{paracontact paracomplex Riemannian manifolds}. In the other case, the manifolds
from the rest of the classes, given in \thmref{thm:para}, we call \emph{paracontact
almost paracomplex Riemannian manifolds}.

\section{The associated Nijenhuis tensor}\label{sec-ANT}


By analogy with the skew-symmetric Lie bracket (the commutator), determined by
${[x,y]=\nabla_x y - \nabla_y x}$, let us consider the symmetric braces (the
anticommutator), defined by
$\{x,y\} = \nabla_x y {+}  \nabla_y x$ as in \cite{MI1}.
Bearing in mind the definition of the Nijenhuis torsion $\left[\check P,\check
P\right]$ of an almost paracomplex structure $\check P$ on $\check\MM$, we give a
definition of a tensor $\left\{\check P,\check P\right\}$ of type $(1,2)$ as follows:
\[
\left\{\check P,\check P\right\}(\check x,\check y)=\left\{\check x,\check
y\right\}+\left\{\check P\check x,\check P \check y\right\} -\check P\left\{\check P
\check x,\check y\right\}- \check P\left\{\check x,\check P \check y\right\},
\]
where the action of $\check P$ is given in (\ref{P}) and the anticommutator on the
tangent bundle of $\check\MM$ is determined by:
\[
\bigl\{\left(x,a\partial_r \right),\left(y,b\partial_r \right)\bigr\} =
\bigl(\{x,y\},\left(x(b)+y(a)\right)\partial_r\bigr).
\]
We call $\left\{\check P,\check P\right\}$ an \emph{associated Nijenhuis tensor of the
almost paracomplex manifold  $\left(\check \MM,\check P\right)$}. Obviously, this
tensor is symmetric with respect to its arguments, i.e.
$\left\{\check P,\check P\right\}(\check x,\check y)=\left\{\check P,\check
P\right\}(\check y,\check x)$.

Since the almost paracomplex manifold $\left(\check \MM,\check P\right)$ is generated
from
the almost paracontact almost paracomplex  manifold $(\MM,\f,\xi,\eta)$, we seek to
express
the associated Nijenhuis tensor $\left\{\check P,\check P\right\}$
by tensors for the structure $(\f,\xi,\eta)$.
Since $\left\{\check P,\check P\right\}$ is a tensor field of type (1,2) on $\check
\MM$, it suffices to compute the following two expressions:
\[
\begin{array}{l}
\left\{\check P,\check P\right\}\bigl((x,0),(y,0)\bigr)=
\bigl\{(x,0),(y,0)\bigr\}
+\bigl\{\check P(x,0),\check P (y,0)\bigr\}
-\check P\bigl\{\check P(x,0),(y,0)\bigr\}\\[0pt]
\phantom{\left\{\check P,\check P\right\}\bigl((x,0),(y,0)\bigr)=
\bigl\{(x,0),(y,0)\bigr\}+\bigl\{\check P(x,0),\check P (y,0)\bigr\}}
-\check P\bigl\{(x,0),\check P(y,0)\bigr\} \\[0pt]
\phantom{\left\{\check P,\check P\right\}\bigl((x,0),(y,0)\bigr\}}
=
\bigl(\{x,y\},0 \bigr)
+\bigl\{(\f x,\eta(x)\ze),(\f y,\eta(y)\ze)\bigr\}
\\[0pt]
\phantom{\left\{\check P,\check P\right\}\bigl((x,0),(y,0)\bigr)=\,}
-\check P\bigl\{(\f x,\eta(x)\ze),(y,0)\bigr\}-\check P\bigl\{(x,0),(\f
y,\eta(y)\ze)\bigr\}\\[0pt]
\phantom{\left\{\check P, \check P\right\}\bigl((x,0),(y,0)\bigr)}
 =\bigl(\{\f,\f\}(x,y)-(\LL_\xi g)(x,y)\xi, \\
\phantom{\left\{\check P,\check P\right\}\bigl((x,0),(y,0)\bigr)=\;\;}
\left((\LL_\xi g)(\f x,y) + (\LL_\xi g)(x,\f y) \right)\ze \bigr)\\[0pt]
\end{array}
\]
and
\[
\begin{array}{l}
\left\{\check P,\check
P\right\}\bigl((x,0),(0,\ze)\bigr)=\bigl\{(x,0),(0,\ze)\bigr\}+\bigl\{\check
P(x,0),\check P (0,\ze)\bigr\}
-\check P\bigl\{\check P(x,0),(0,\ze)\bigr\}\\[0pt]
\phantom{\left\{\check P,\check
P\right\}\bigl((x,0),(0,\ze)\bigr)=\bigl\{(x,0),(0,\ze)\bigr\}+\bigl\{\check
P(x,0),\check P (0,\ze)\bigr\}}
-\check P\bigl\{(x,0),\check P(0,\ze)\bigr\} \\[0pt]
\phantom{\left\{\check P,\check P\right\}\bigl((x,0),(y,0)\bigr)}
=\bigr\{(\f x,\eta(x)\ze),(\xi,0)\bigr\} -\check P\bigl\{(\f
x,\eta(x)\ze),(0,\ze)\bigr\}\\
\phantom{\left\{\check P,\check P\right\}\bigl((x,0),(y,0)\bigr)=\bigr\{(\f
x,\eta(x)\ze),(\xi,0)\bigr\}}
-\check P\bigl\{(x,0),(\xi,0)\bigr\}\\[0pt]
\phantom{\left\{\check P,\check P\right\}\bigl((x,0),(y,0)\bigr)}
=\left(\{\f x,\xi\}-\f\{x,\xi\}, (\LL_\xi g)(x,\xi)\ze\right).
\end{array}
\]
In the latter expressions, we use the Lie derivative $(\LL_\xi g)(x,y)$, determined by
(\ref{L}), of the Riemannian metric $g$ of $\M$.

Then, we define the following four tensors $\hatN^{(k)}$ ${(k=1,2,3,4)}$ of type
(1,2), (0,2), (1,1), (0,1), respectively:
\begin{equation}\label{HNK}
\begin{array}{l}
\hatN^{(1)}(x,y)=\{\f,\f\}(x,y) -(\LL_\xi g) (x,y)\xi,\\[0pt]
\hatN^{(2)}(x,y)=(\LL_\xi g)(\f x,y)+(\LL_\xi g)(x,\f y),\\[0pt]
\hatN^{(3)}(x)=\{\f x,\xi\}-\f\{x,\xi\}
-\f\nabla_x \xi
,\\[0pt]
\hatN^{(4)}(x)=
(\LL_\xi g)(x,\xi),
\end{array}
\end{equation}
where $\{\f ,\f\}$ is the symmetric tensor of type (1,2) determined by:
\begin{equation}\label{{ff}}
\{\f ,\f\}(x,y)=\{\f x,\f y\}+\f^2\{x,y\}-\f\{\f x,y\}-\f\{x,\f y\}.
\end{equation}

By direct converting their definitions, we find relations between the four tensors
$\hatN^{(k)}$ as follows:
\begin{equation}\label{hatNkN}
	\begin{array}{l}
 \hatN^{(2)}(x,y) =
-\eta \left( \hatN^{(1)}(x,\f y)\right) -\eta \left( \hatN^{(1)}(\f
x,\xi)\right)\eta(y),\\[0pt]
\hatN^{(3)}(x) = \hatN^{(1)}(\f x,\xi)- \eta(x)\f \hatN^{(1)}(\xi,\xi)\\[0pt]
\hatN^{(4)}(x)=-\eta\left(\hatN^{(1)}(x,\xi)\right)
=\frac12g\left(\hatN^{(1)}(\xi,\xi),x\right),\\[0pt]
\hatN^{(4)}(x) = \hatN^{(2)}(\f x,\xi),\qquad
\hatN^{(4)}(x) = -\eta \left( \hatN^{(3)}(\f x)\right).
\end{array}
\end{equation}

\begin{thm}\label{thm-NHK}
For an almost paracontact almost paracomplex Riemannian manifold we have:
\begin{enumerate}
\item[a)] if $\hatN^{(1)}$ vanishes, then all the other tensors $\hatN^{(2)}$,
    $\hatN^{(3)}$ and $\hatN^{(4)}$ vanish;
\item[b)] if any one of $\hatN^{(2)}$ and $\hatN^{(3)}$ vanishes, then $\hatN^{(4)}$
    vanishes.
\end{enumerate}
\end{thm}
\begin{proof}
The statements above are consequences of the relations (\ref{hatNkN}) between
$\hatN^{(k)}$ $(k=1,2,3,4)$.
\end{proof}

Therefore, $\hatN^{(1)}$ plays a main role between them and we denote it simply by
$\hatN$, i.e.
\begin{equation}\label{HatNdef}
\hatN(x,y) = \{\f,\f\}(x,y) {-} ( \mathcal \LL_{\xi}g)(x,y) \xi
\end{equation}
and we call it an \emph{associated Nijenhuis tensor of the structure
$(\f,\xi,\eta,g)$}.
Obviously, $\hatN$ is symmetric, i.e.  $\hatN(x, y) = \hatN(y, x)$. Applying the
expressions (\ref{{ff}}), (\ref{L}) and (\ref{HatNdef}), the associated Nijenhuis
tensor has the following form in terms of $\nabla \f$ and $\nabla \eta$:
\[
\hatN(x,y)=(\nabla_{\f x} \f)y+(\nabla_{\f y} \f)x-\f(\nabla_ x \f)y-\f(\nabla_ y
\f)x-(\nabla_x \eta)y\,\xi-(\nabla_y \eta)x\,\xi.
\]

The corresponding tensor of type (0,3) is defined by
$\hatN(x,y,z)=g\left(\hatN(x,y),z\right)$. According to (\ref{F=nfi}) and
(\ref{n_eta_F}), we express $\hatN$  in terms of the fundamental tensor $F$ as
follows:
\begin{equation}\label{HN}
\begin{array}{ll}
\hatN(x,y,z)=F(\f x,y,z)+F(\f y,x,z)-F(x,y,\f z)-F(y,x,\f z)\\[0pt]
\phantom{\hatN(x,y,z)=}
+\eta(z)\left\{F(x,\f y,\xi)+F(y,\f x,\xi)\right\}.
\end{array}
\end{equation}

\begin{prop}\label{prop-HN}
The associated Nijenhuis tensor on an almost paracontact almost paracomplex Riemannian
manifold has the following properties:
\[
\begin{array}{ll}
\hatN(\f^2 x, \f y,\f z)= - \hatN(\f^2 x, \f^2 y,\f^2 z), \qquad& \hatN(\f^2 x, \f^2
y,\f^2 z) = \hatN(\f x, \f y,\f^2 z), \\[0pt]
\hatN( x, \f^2 y,\f^2 z)= - \hatN(x, \f y,\f z), \qquad& \hatN(\f^2 x, \f^2 y, z) =
\hatN(\f x, \f y, z),\\[0pt]
\hatN(\xi, \f y,\f z) = - \hatN(\xi, \f^2 y,\f^2 z), \qquad& \hatN(\f x, \f y, \xi) =
\hatN(\f^2 x, \f^2 y, \xi).
\end{array}
\]
\end{prop}
\begin{proof}
The results follow form the properties (\ref{F-prop}) of $F$ and the expression
(\ref{HN}).
\end{proof}

Applying  (\ref{HN}) to  (\ref{hatNkN}), we give the form of $\hatN^{(2)}$,
$\hatN^{(3)}$ and $\hatN^{(4)}$ in terms of $F$:
\begin{equation}\label{hatNkF}
	\begin{array}{l}
 \hatN^{(2)}(x,y) = - F(x,y,\xi) - F(y,x,\xi) - F(\f x,\f y,\xi) - F(\f y,\f
 x,\xi),\\[0pt]
 \hatN^{(3)}(x, y)= F(\xi,x,y) + F(x,y,\xi) - F(\f x,\f y,\xi),\\[0pt]
 \hatN^{(4)}(x) = - F(\xi, \f x,\xi),
\end{array}
\end{equation}
where we use the denotation $\hatN^{(3)}(x, y) = g\left(\hatN^{(3)}(x), y\right)$.

\begin{prop}\label{prop-N3}
Let $(\MM,\f,\xi,\eta,g)$ be an $\F_{i}$-manifold $(i=1,2,\dots,11)$. Then the four
tensors $\hatN^{(k)}$  ${(k=1,2,3,4)}$ on this manifold have the form in the
respective cases, given in \tablref{tab:AssociatedNijenhuisTensors}.
\begin{table}
 \caption{Associated Nijenhuis tensors}\label{tab:AssociatedNijenhuisTensors}
  \vglue4mm
\centering
\footnotesize{
\begin{tabular}{|c||c|c|c|c|}
\hline
& ${\hatN^{(1)}(x,y,z)}$ & $\hatN^{(2)}(x,y)$ & $\hatN^{(3)}(x,y)$ & $\hatN^{(4)}(x)$
\\
\hline
\hline
$\F_{1}$ & $\frac{2}{n}\bigl\{g(x,\f y)\ta(\f^2 z)-g(\f
x,\f y)\ta(\f z)\bigr\}$& {0} & {0} & {0}\\
$\F_{2}$ & $-2\bigl\{F(\f x,\f y,\f z)+F(\f^2 x,\f^2 y,\f z)\bigr\}$& {0} & {0} &
{0}\\
$\F_{3}$ & {0}& {0} & {0} & {0}\\
$\F_{4}$ & $\frac{2}{n}\theta(\xi)g(x,\f y)\eta(z)$& $-\frac{2}{n}\theta(\xi)g(\f x,\f
y)$ & {0} & {0}\\
$\F_{5}$ & $\frac{2}{n}\theta^*(\xi)g(\f x,\f y)\eta(z)$& $-\frac{2}{n}\theta^*(\xi)g(
x,\f y)$ & {0} & {0}\\
$\F_{6}$ & $4F(x,\f y,\xi)\eta(z)$& $-4F(x,y,\xi)$ & {0} & {0}\\
$\F_{7}$ & {0}& {0} & {0} & {0}\\
$\F_{8}$ & $-2\bigl\{\eta(x)F(y,\f z,\xi) +\eta(y)F(x,\f
z,\xi)\bigr\}$& {0} & $2F(x,y,\xi)$ & {0}\\
$\F_{9}$ & $-2\bigl\{\eta(x)F(y,\f z,\xi) +\eta(y)F(x,\f
z,\xi)\bigr\}$& {0} & $2F(x,y,\xi)$ & {0}\\
$\F_{10}$ & $-\eta(x)F(\xi,y,\f z)-\eta(y)F(\xi,x,\f
z)$ & {0}& $F(\xi,x,y)$ & {0} \\
$\F_{11}$ & $\eta(z)\bigl\{\eta(x)\om(\f y)+\eta(y)\om(\f
x)\bigr\}$ & $-\eta(x)\om(y)-\eta(y)\om(x)$& $\om(x)\eta(y)+2\eta(x)\om(y)$ & $-\om(\f
x)$ \\
$$ & $- 2\eta(x)\eta(y)\om(\f z)$ & $$& $$ & $$ \\
\hline
\end{tabular}}	
\end{table}
\end{prop}
\begin{proof}
The calculations are made, using (\ref{HN}), (\ref{hatNkF}) and the expression
(\ref{F1-11}) of each of $F_i$ for the corresponding class $\F_i$.
\end{proof}

As a result of \propref{prop-N3}, we establish the truthfulness of the following
\begin{thm}\label{thm-HN}
An almost paracontact almost paracomplex Riemannian manifold $\M$ has:
\begin{enumerate}
  \item[a)]
 vanishing $\hatN^{(1)}$ if and only if it belongs to some of the basic classes
 $\F_3$, $\F_7$ or to their direct sum;
  \item[b)]
 vanishing $\hatN^{(2)}$ if and only if it belongs to some of the basic classes
 $\F_1$,  $\F_2$,  $\F_3$, $\F_7$, $\dots$, $\F_{10}$ or to their direct sums;
  \item[c)]
  vanishing $\hatN^{(3)}$ if and only if it belongs to some of the basic classes
  $\F_1,\dots, \F_7$ or to their direct sums;
    \item[d)]
  vanishing $\hatN^{(4)}$ if and only if it belongs to some of the basic classes
  $\F_1,\dots, \F_{10}$ or to their direct sums.
\end{enumerate}
\end{thm}
By virtue of \thmref{thm-HN}, we obtain the following
\begin{cor}
The class of almost paracontact almost paracomplex Riemannian manifolds with a
vanishing associated Nijenhuis tensor $\hatN$ is $\F_3\oplus\F_7$.
\end{cor}

\section{The pair of Nijenhuis tensors and the classification of the considered
manifolds}\label{sec-FNhaN}

In the previous two sections, by (\ref{N}) and (\ref{HN}), we give the expressions of
the Nijenhuis tensor $N$ and its associated $\hatN$ by the tensor $F$, respectively.
Here, we find how the fundamental tensor $F$ is determined by the pair of Nijenhuis
tensors. Since $F$ is used for classifying the manifolds studied, we can expressed the
classes $\F_i$ $i=(1,2,\dots,11)$ only by the pair $(N,\hatN)$.

\begin{thm}\label{gen}
Let $\M$ be an almost paracontact almost paracomplex Rie\-mannian manifold. Then its
fundamental tensor is expressed by $N$ and $\hatN$ by the
formula:
\begin{equation}\label{nabf}
\begin{array}{l}
F(x,y,z)
=\frac14\bigl[N(\f x,y,z)+N(\f x,z,y)
+\hatN(\f x,y,z)+\hatN(\f x,z,y)\bigr]\\[0pt]
\phantom{F(x,y,z)=}
-\frac12\eta(x)\bigl[N(\xi,y,\f z)+\hatN(\xi,y,\f z)+\eta(z)\hatN(\xi,\xi,\f
y)\bigr].
\end{array}
\end{equation}
\end{thm}
\begin{proof}
Taking the sum of (\ref{N}) and (\ref{HN}), we obtain:
\begin{equation}\label{sum1}
F(\f x,y,z)-F(x,y,\f z)=\frac12\bigl[N(x,y,z)
+\hatN(x,y,z)\bigr]-\eta(z)F(x,\f y,\xi).
\end{equation}
The  identities  (\ref{F-prop}) together with (\ref{str}) imply:
\begin{equation}\label{sum2}
F(x,y,\f z)+F(x,z,\f y)=\eta(z)F(x,\f y,\xi)+\eta(y)F(x,\f
z,\xi).
\end{equation}
A suitable combination of (\ref{sum1}) and (\ref{sum2}) yields:
\begin{equation}\label{nabff}
F(\f x,y,z)=\frac14\bigl[N(x,y,z)+N(x,z,y)+
\hatN(x,y,z)+\widehat N(x,z,y)\bigr]. 
\end{equation}
Applying (\ref{str}), we obtain from (\ref{nabff}) the following:
\begin{equation}\label{nabff1}
\begin{array}{ll}
F(x,y,z)=\frac14\bigl[N(\f x,y,z)+N(\f x,z,y)+\hatN(\f x,y,z)+\hatN(\f
x,z,y)\bigr]\\[0pt]
\phantom{F(x,y,z)=}
+\eta(x)F(\xi,y,z).
\end{array}
\end{equation}
Set $x=\xi$ and $z\rightarrow \f z$ into (\ref{sum1}) and use
(\ref{str}) to get:
\begin{equation}\label{nabff2}
F(\xi,y,z)=-\frac12\bigl[N(\xi,y,\f z)+\hatN(\xi,y,\f
z)\bigr]+\eta(z)\om(y).
\end{equation}
Finally, using (\ref{HN}) and the general
identities $\om(\xi)=0$, we obtain:
\begin{equation}\label{fff}
\om(z)=-\frac12\hatN(\xi,\xi,\f z).
\end{equation}
Substitute (\ref{fff} )into (\ref{nabff2}) and the obtained
identity insert into (\ref{nabff1}) to get (\ref{nabf}).
\end{proof}

\begin{cor}\label{cor:FN}
The class of almost paracontact almost paracomplex Rie\-mannian manifolds with
vanishing tensors $N$ and $\hatN$ is the special class $\F_0$.
\end{cor}

\section{A family of Lie groups as manifolds of the studied type}\label{sec-exm}

Let $\LLL$ be a $(2n+1)$-dimensional real connected Lie group and let its
associated Lie algebra with a global basis $\{E_{0},E_{1},\dots,
E_{2n}\}$ of left invariant vector fields on $\LLL$ be defined by:
\begin{equation}\label{com}
    [E_0,E_i]=-a_iE_i-a_{n+i}E_{n+i},\qquad
    [E_0,E_{n+i}]=-a_{n+i}E_i+a_{i}E_{n+i},
\end{equation}
where $a_1,\dots,a_{2n}$ are real constants and $[E_j,E_k]=0$ in
other cases.

Let $(\f,\xi,\eta)$ be an almost paracontact almost paracomplex structure determined
for any
${i\in\{1,\dots,n\}}$ by:
\begin{equation}\label{strL}
\begin{array}{lll}
\f E_0=0,\qquad & \f E_i=E_{n+i},\qquad & \f E_{n+i}=E_i, \\[0pt]
\xi=E_0, \qquad  & \eta(E_0)=1, \qquad & \eta(E_i)=\eta(E_{n+i})=0.
\end{array}
\end{equation}
Let $g$ be a Riemannian metric defined by:
\begin{equation}\label{gL}
\begin{array}{l}
  g(E_0,E_0)=g(E_i,E_i)=g(E_{n+i},E_{n+i})=1, \\
  g(E_0,E_j)=g(E_j,E_k)=0,
\end{array}
\end{equation}
where $i\in\{1,\dots,n\}$ and $j, k \in\{1,\dots,2n\}$, $j\neq k$.
Thus, since (\ref{str}) is satisfied, the induced $(2n+1)$-dimensional
manifold $(\LLL,\f, \xi, \eta, g)$ is an almost paracontact almost paracomplex
Riemannian manifold.

Let us remark that in \cite{Ol} the same Lie group is considered with an appropriate
almost contact structure and a compatible Riemannian
metric. Then, the generated almost cosymplectic manifold is studied.
On the other hand,
in \cite{HM}, the same Lie group is equipped with an almost contact structure and
B-metric. Then, the obtained manifold is characterized. Moreover, in \cite{HManMek},
the case of the lowest dimension is considered and properties of the constructed
manifold are determined.

Let us consider the constructed almost paracontact almost paracomplex Riemannian
manifold
$(\LLL,\f, \xi, \allowbreak{}\eta, g)$ of dimension 3, i.e. for $n=1$.

According to (\ref{com}) and (\ref{gL}) for $n=1$,  by the
Koszul equality
\[
2g\left(\n_{E_i}E_j,E_k\right)
=g\left([E_i,E_j],E_k\right)+g\left([E_k,E_i],E_j\right)
+g\left([E_k,E_j],E_i\right)
\]
for the Levi-Civita connection $\n$ of $g$, we obtain:
\begin{equation}\label{nEi}
\begin{array}{ll}
    \n_{E_1}E_0=a_1E_1+a_2E_2,\qquad & \n_{E_2}E_0=a_2E_1-a_1E_2,\\[0pt]
    \n_{E_1}E_1=-\n_{E_2}E_2=-a_1E_0,\qquad & \n_{E_1}E_2=\n_{E_2}E_1=-a_2E_0,
\end{array}
\end{equation}
and the others $\n_{E_i}E_j$ are zero.

Then, using  (\ref{nEi}), (\ref{strL}), (\ref{F=nfi}) and (\ref{Fi3}),
we get the following  components $F_{ijk}=F(E_i,E_j,E_k)$
of the fundamental tensor:
\[
F_{101}=F_{110}=F_{202}=F_{220}=-a_2,\qquad
F_{102}=F_{120}=-F_{201}=-F_{210}=-a_1,
\]
and the other components of $F$  are zero.
Thus, we have  the expression of $F$ for arbitrary vectors
$x=x^iE_i$, $y=y^iE_i$, $z=z^iE_i$ as follows:
\begin{equation}\label{Fexa}
\begin{array}{ll}
 F(x,y,z)={-}a_2\left\{x^1\left(y^0z^1+y^1z^0\right)+x^2\left(y^0z^2+y^2z^0\right)\right\}\\[0pt]
\phantom{ F(x,y,z)=\,\,}
{-}a_1\left\{x^1\left(y^0z^2+y^2z^0\right)-x^2\left(y^0z^1+y^1z^0\right)\right\}.
\end{array}
\end{equation}
Bearing in mind the latter equality, we obtain that $F$ has  the following form:
\[
F(x,y,z)=F_4(x,y,z)+F_9(x,y,z),
\]
by virtue of (\ref{Fi3}) for $\mu=-a_1$, $\ta_0=-2a_2$. Therefore, we have proved the
following

\begin{prop}\label{prop-exa1}
The constructed 3-dimen\-sional almost paracontact almost paracomplex Riemannian
manifold $(\LLL,\f,\xi,\eta,g)$ belongs to:
\begin{enumerate}
  \item[a)]
 $\F_4\oplus\F_9$ if and only if $a_1\neq 0$, $a_2\neq 0$;
  \item[b)]
 $\F_4$ if and only if $a_1= 0$, $a_2\neq 0$;
  \item[c)]
  $\F_9$ if and only if $a_1\neq 0$, $a_2= 0$;
    \item[d)]
  $\F_0$ if and only if $a_1= 0$, $a_2= 0$.
\end{enumerate}
\end{prop}

Finally, we get the following
\begin{prop}\label{prop-exa1a}
The constructed 3-dimensional almost paracontact almost paracomplex Riemannian
manifold $(\LLL,\f,\allowbreak{}\xi,\eta,g)$ has the following properties:
\begin{enumerate}
  \item[a)]	
	It has vanishing $N^{(4)}$ and $\hatN^{(4)}$;
  \item[b)]
 It is a normal almost paracontact almost paracomplex Riemannian manifold with
 vanishing $\hatN^{(3)}$ if and only if $a_1= 0$ and arbitrary $a_2$;
%
$a_1= 0$ and arbitrary $a_2$;
	
$a_2\neq 0$;
  \item[c)]
 It is a para-Sasakian paracomplex Riemannian manifold if and only if $a_1= 0$,
 $a_2=1$;
  \item[d)]
 It has vanishing $\hatN^{(2)}$ if and only if $a_2= 0$ and arbitrary $a_1$.
\end{enumerate}
\end{prop}
\begin{proof}
According to (\ref{Fexa}), (\ref{Fi3}) and \propref{prop-p3}, we find the following
form of the Nijenhuis tensor of $(\LLL,\f,\allowbreak{}\xi,\eta,g)$:
\[
N(x,y,z)=-2a_1\left\{\left(x^1y^2-x^2y^1\right)z^0+\left(x^0y^1-x^1y^0\right)z^1-\left(x^0y^2-x^2y^0\right)z^2\right\}.
\]
From the latter equality and (\ref{NkN}) (or alternatively from (\ref{Fexa}) and
(\ref{NkF})), we have:
\[
N^{(2)}(x,y)=2a_1(x^1y^1-x^2y^2),\quad
N^{(3)}(x,y)=2a_1(x^1y^2-x^2y^1),\quad
N^{(4)}(x)=0.
\]

Similarly, for the associated Nijenhuis tensor of $(\LLL,\f,\allowbreak{}\xi,\eta,g)$
we obtain:
\[
\hatN(x,y,z)=-4a_2\left(x^1y^2+x^2y^1\right)z^0
+2a_1\left\{\left(x^0y^1+x^1y^0\right)z^1-\left(x^0y^2+x^2y^0\right)z^2\right\}.
\]
By virtue of (\ref{hatNkN}) and the equality from above (or in other way by
(\ref{Fexa}) and (\ref{hatNkF})), we get:
\[
\hatN^{(2)}(x,y)=4a_2\left(x^1y^1+x^2y^2\right), \quad
\hatN^{(3)}(x,y)=-2a_1\left(x^1y^2-x^2y^1\right),\quad
\hatN^{(4)}(x)=0.
\]

As a conclusion, the obtained results imply the propositions in a), b) and d).
Moreover, the case of the ${\F_4}'$-manifold, i.e. the proposition in c), follows from
\propref{prop-exa1} b).
\end{proof}

 \bigskip

{\small\rm\baselineskip=10pt
 \baselineskip=10pt
 \qquad Mancho H. Manev \par
 \qquad Faculty of Mathematics and
Informatics, University of Plovdiv Paisii Hilendarski \par
 \qquad Department of Algebra and Geometry \par
 \qquad 24 Tzar Asen St, 4000 Plovdiv,
Bulgaria \par
 \qquad and \par
\qquad  Faculty of Public
Health, Medical University of Plovdiv \par
 \qquad Department of Medical Informatics, Biostatistics and E-Learning \par
 \qquad 15A 	Vasil Aprilov Blvd, 4002 Plovdiv,
Bulgaria \par
 \qquad {\tt mmanev@uni-plovdiv.bg}

 \bigskip \smallskip

 \qquad Veselina R. Tavkova \par
 \qquad Faculty of Mathematics and
Informatics, University of Plovdiv Paisii Hilendarski\par
 \qquad Department of Algebra and Geometry\par
 \qquad 24 Tzar Asen St, 4000 Plovdiv,
Bulgaria\par
 \qquad {\tt vtavkova@uni-plovdiv.bg}
 }

 \end{document}